\RequirePackage{fix-cm}
\documentclass[smallextended]{svjour3}

\smartqed  
\usepackage{graphicx}
\usepackage{gensymb}
\usepackage[square,sort,comma,numbers]{natbib}
\usepackage{amssymb} 
\usepackage{amsthm}
\usepackage{amsmath}

\usepackage{mathtools}   
\usepackage{graphicx}
\usepackage{graphics}  
\usepackage{color}
\usepackage{verbatim} 
\usepackage[normalem]{ulem} 
\usepackage[mathscr]{eucal}
\usepackage{upgreek}
\usepackage{enumerate} 
\usepackage{bbm}
\usepackage{dsfont}
\numberwithin{equation}{section}
\theoremstyle{plain}

\newcommand{\babs}[1]{{\bigl\lvert #1\bigr\rvert}}
\newcommand{\Babs}[1]{{\Bigl\lvert #1\Bigr\rvert}}

\DeclarePairedDelimiter{\abs}{\lvert}{\rvert}
\DeclarePairedDelimiter{\norm}{\lVert}{\rVert}

\renewcommand{\d}{{\rm d}} 
\newcommand{\N}{{\mathbb{N}}}
\newcommand{\E}{{\mathbb{E}}}

\newcommand{\R}{{\mathbb{R}}}
\newcommand{\Z}{{\mathbb{Z}}}
\renewcommand{\P}{{\mathbb{P}}}

\newcommand{\Acal}   {{\mathcal A}} 
\newcommand{\Ccal}   {{\mathcal C }} 
\newcommand{\Ecal}   {{\mathcal E }} 
\newcommand{\Jcal}   {{\mathcal J }} 
\newcommand{\Kcal}   {{\mathcal K }} 
\newcommand{\Rcal}   {{\mathcal R }} 
\newcommand{\Tcal}   {{\mathcal T }} 
\newcommand{\Ocal}   {{\mathcal O }} 
\newcommand{\Mcal}   {{\mathcal M }} 
\newcommand{\Wcal}   {{\mathcal W }} 
\newcommand{\Scal}{{\mathcal{S}}}

\newcommand{\1} {\mathds{1}} 
\newcommand{\e} {\varepsilon} 
\newcommand{\ex}{{\rm e}} 
\newcommand{\eps}{\varepsilon}

\newcommand{\Sfrak}{{\mathfrak{S}}}
\newcommand{\p}{\mathfrak{p}}
\renewcommand{\L} {\Lambda} 
\renewcommand{\l} {\lambda} 

\newcommand{\Ttt}{{T_\tau^{1/2}}}
\newcommand{\pn}{p^{\langle N \rangle}}
\newcommand{\Pn}{P^{\langle N \rangle}}
\newcommand{\En}{E^{\langle N \rangle}}
\newcommand{\Dt}{\Delta^\tau_N}
\newcommand{\Ene}{E^{\langle N \rangle}_{n,\e}}
\newtheorem{cor}[theorem]  {Corollary} 
\newcommand{\Pne}{P^{\langle N \rangle}_{n,\e}}
\newcommand{\Tt}{{T_\tau}}
\newcommand{\floor}[1]{\left\lfloor #1 \right\rfloor}
\begin{document}

\title{Large deviations of the range of the planar random walk on the scale of the mean
}

\titlerunning{Large deviations of the range of the planar random walk}        

\author{Jingjia Liu         \and
       Quirin  Vogel* 
}


\institute{Jingjia Liu \at
             {Fachbereich Mathematik und Informatik,
	Universit\"at M\"unster,
	Einsteinstra\ss e 62,
	48149 M\"unster,
	Germany}\\
              \email{jingjia.liu@uni-muenster.de}           
           \and
           Quirin Vogel* \at
              {Mathematics Institute, University of Warwick, Coventry CV4 7AL, United Kingdom}\\
              \email{Q.Vogel@warwick.ac.uk*}  
}

\date{Received: date / Accepted: date}

\maketitle

\begin{abstract}
    We prove an upper large deviation bound on the scale of the mean for a symmetric random walk in the plane satisfying certain moment conditions. This result complements the study of van den Berg, Bolthausen and den Hollander, where the continuum case of the Wiener Sausage was studied, and in Phetpradap, in which one is restricted to dimension three and higher.
\end{abstract}

\section{Introduction}
The range of the random walk is a topic which has been studied for more than 60 years. For our purpose, a random walk $S_n=\sum_{i=1}^n X_i$ is defined to be the sum of i.i.d.\ random variables on $\Z^d$, where $X_i$ has mean zero. The range $\Rcal_n$ of a random walk is then defined as
\begin{equation*}
    \Rcal_n=\# \{x\in \Z^d\colon \exists\, k\in \{0,\ldots n\}\text{ with } S_k=x\}\, .
\end{equation*}
Previous works, for example \cite{dvoretzky1951} and \cite{Jain1971}, examined the mean and the variance of $\Rcal_n$. It was proven that the mean range of a fairly general random walk (with identity covariance) is given by 
\begin{equation}\label{meanEq}
    E[\Rcal_n]\sim\begin{cases}
    \left(\frac{8n}{\pi}\right)^{1/2}&\text{ if }d=1\, ,\\
    \frac{2\pi n}{\log n}&\text{ if }d=2\, , \\
    \kappa_d n &\text{ if } d\ge 3\, ,
    \end{cases}
\end{equation}
with $\kappa_d=P(S_i\neq 0 \text{ for all } i\ge 1)$, see \cite[Equation 5.3.39]{chen2010random}\footnote{Note that a variety of probability measures is being used throughout the text. As a rule of thumb, \texttt{mathbb} denotes a measure on continuum objects, while standard font refers to discrete random variables. For more details, see Section 5.}.  \\
Estimations of error terms as well as asymptotics for the variance of $\Rcal_n$ are also available from the aforementioned references. As proof techniques developed, large deviation results on the scale $n^{d/(d+2)}$ were obtained in \cite{DV75DistinctRW}. Central limit theorems were given in \cite{JO68} and \cite{gall1991}. Laws of the iterated logarithm can be found in \cite{BK02}. For a good overview of these classical results, we refer the reader to \cite{chen2010random}, where more precise statements are presented.

In recent years the understanding of different properties of the range of the random walk had been refined. We present a (very incomplete) selection of these results. In \cite{bass2009moderate} moderate deviations of the renormalized range $\Rcal_n-\E[\Rcal_n]$ were studied. Uchiyama gave in \cite{Uchiyama2009} an asymptotic expansion of the expectation of the range of a random walk bridge, which holds uniformly in a large set of possible end-points of the bridge. In a sequence of papers \cite{AS}, \cite{ASS1} and \cite{ASS2}, the capacity of the range of the random walk was analyzed, with a focus on precise results in high dimensions. A strong law of large numbers type result for the boundary of the range of the random walk was obtained in \cite{DGK18} for both transient and recurrent random walks. Additionally, the range of a planar random walk conditioned on never hitting the origin was studied in \cite{GPV18}.

The study of the range of random walks has a lot of important applications. For instance, it gives us a glimpse into the geometry and approximate fractal dimension of the sample paths. Furthermore, the range of random walks can be used in the derivation of an asymptotic expansion for the random walk on a lattice with a random distribution of traps, see \cite{denHollander1984}. It can be also used to describe the volume fraction of polymers, see \cite{WR07}. Another reason to research the range of random walks arises from its interesting connection to the Gaussian free field. More precisely, the isomorphism constructed by Le Jan in \cite{lejan2008} allows us to express certain connectivity properties of the field in terms of the range of the random walk. Hence we hope that our work is not only essential in its own light by providing a concentration inequality and allowing the evaluation of exponential functionals of the range, but also serves as a tool to improve our understanding of other areas of probability theory.

The continuum analogue of the range of a random walk can be seen as the volume along the trajectory of a Brownian motion: Denote a standard Brownian motion in $\R^d$ by $(M_t)_{t\geq 0}$, and define the Wiener sausage ($a>0$)
$$
W_t^a=\bigcup_{0\leq s\leq t} B_a(M_s),\quad t\geq 0,
$$
where $B_a(x)$ is the ball with radius $a$ around a point $x$. In \cite{BBH}, a large deviation bound was proven
\begin{equation}\label{BBHbound}
    \lim_{t\to \infty} \frac{1}{\tau}\log \P\left(\texttt{volume}(W_t^a)\le b\Tt \right)=-I^{g_a}(b)\, ,
\end{equation}
where $\tau$ is given by $t^{(d-2)/d}$ for $d\ge 3$ and $\tau=\log t$ for $d=2$. The scale of the mean is $\Tt=t$, for $d\ge 3$ and $\Tt=t/\log t$, for $d=2$. Here, $I^{g_a}$ is the rate function depending on $g_a$, a (dimension dependent) constant parameterized by the thickening $a>0$ in the case $d\ge 3$. For $d=2$ one has $g_a=2\pi$ and thus is independent of $a>0$. The rate function can be computed explicitly in terms of the minimizer of a variational problem. Properties of the rate function are given in the aforementioned paper as well and we quote some of them in Appendix, see Theorem \ref{ThmPropRateFunct}.
Motivated by \cite{BBH}, similar results have been deduced for range of random walks: Phetpradap in his PhD thesis \cite{PP11} proved an equivalent version to Equation \eqref{BBHbound} by replacing $\texttt{volume}(W_t^a)$ with $\Rcal_n$ in the case of the simple symmetric random walk in $d\geq 3$:

\begin{theorem}\cite{PP11}
\label{PPthm}
Let $d\geq 3$ and $\Rcal_n$ be the range of the symmetric nearest neighbor random walk in $\Z^d$. For every $b>0$,
\begin{align}\label{discrete3}
\lim_{n\to\infty}\frac{1}{n^{\frac{d-2}{d}}}\log P\left(\Rcal_n\leq bn \right)=-\frac{1}{d}I^\kappa(b),
\end{align}
where
\begin{align*}
I^\kappa(b)=\inf_{\phi\in\Phi^\kappa(b)}\left\lbrack \frac{1}{2}\int_{\R^d} |\nabla\phi|^2(x)\mathrm{d}x \right\rbrack
\end{align*}
with
\begin{align*}
\Phi^\kappa(b)=\left\lbrace \phi\in H^1(\R^d)\colon \int_{\R^d} \phi^2(x) \mathrm{d}x=1,\;\int_{\R^d} (1-e^{-\kappa\phi^2(x)})\ \mathrm{d}x\leq b \right\rbrace\, ,
\end{align*}
where $H^1$ is the Sobolev space of square integrable functions with a weak first derivative. Here $\kappa$ is the non-return probability of random walk on $\Z^d$
$$
\kappa:=\kappa_d = P(S_i\neq 0\text{ for all } i\ge 1).
$$
\end{theorem}

Our work closes the gap and gives a large deviation type bound for the range of a random walk in the case $d=2$. 
In order to prove the equivalent result for $d=2$, our work is significantly inspired by the previous two publications. However, in contrast to the assumptions in \cite{PP11}, we only require moment bounds on $X_i$. We need additional approximations in comparison to \cite{BBH} to account for the discreteness of $\Rcal_n$. Although this paper is written for the case $d=2$, the methods carry over to $d\ge 3$ (although in that case one has to assume the finiteness of the random walk moment generating function on $\R^d$) and thus slightly generalize the above given statement from \cite{PP11} to more general random walks. With our techniques one can also prove a large deviation result for the intersection of random walks, see the note \cite{Qu20}.\\
Unlike in the case $d\geq 3$ in \cite{PP11}, the a priory lack of boundedness of $\Rcal_n/\Tt$ leads to difficulties in proof, e.g. controlling the moment generating function of $\Rcal_n/\Tt$. We make use of recent results, in particular \cite{uchiyama2011} and \cite{chen2010random} in the course of tackling this issue. 

Lower bounds for Equation \eqref{BBHbound} on the scale of the mean for the Equation \eqref{BBHbound} as well as in the discrete case \eqref{discrete3} have been obtained in \cite{HK01} for $d\ge 3$. Not only does the probability decay on a different scale, but also their techniques differ widely from ours. We thus refer the reader to the paper for more details. It is work in progress to generalize our results to the range of a random walk bridge in the spirit of \cite{HW88} and \cite{hamana2006}.

We organize our paper as follows: In Section \ref{sctn:MainRslt}, we state our large deviation bound on the range of the planar random walk. In the next section, we give the proof, which is organized across different subsections: We begin with a basic large deviation result building on the Donsker-Varadhan theory and then, through a series of approximations and compactifications, arrive at the main result. In the Appendix, we prove various technical lemmas regarding the random walk on the torus. The different approximations, domains and scalings lead to numerous notations, in order to make it easier to follow these, in Section \ref{glossary} we have compiled a list of all the different expressions used.

\section{Main Result and Setting}\label{sctn:MainRslt}
Let $\Rcal_n$ be the range of a random walk $S_n=\sum_{i=1}^n X_i$, where $X_i=(X_i^1,X_i^2)$ are i.i.d. symmetric random variables on $\Z^2$. Denote the distribution of this random walk started at $0$ by $P_0$ and its expectation by $E_0$. We make the following assumptions on $(X_i)_{1\leq i\leq n }$:
\begin{itemize}
    \item  \textbf{Normalization}:
the increments $(X_k)_k$ have mean 0 and the identity as covariance, i.e.\ $E[X^{i}_1X^{j}_1]=\delta_{i}(j)$ for all $1\leq i,j \leq 2$, where $X^{i}_1$ denotes the $i$-th coordinate of the increment $X_1$.
\item  
\textbf{Bounded moments:} let $H\colon [0,\infty)\to (0,\infty)$ be a continuous and increasing function satisfying
\begin{equation}\label{MomentAss}
    \lim_{n\to\infty}\frac{1}{\log n}\log{H(n)}=\infty\, .
\end{equation}
The technical assumption that $H(t)t^{-3-s}$ is eventually increasing (for some $s>0$) and $t^{-1/2}\log H(t) $ is eventually non-increasing, is also needed. We then require that 
$$
E\left[H\left(\abs{X_1}\right)\right]<\infty\, ,
$$
for at least one such $H$.

\item 
\textbf{{Random walk}:}
the random walk $(S_n)_n$ is aperiodic.
\end{itemize}

\begin{remark}
For those $H$ mentioned in moments condition, one can take for example $H(n)=\exp\left(\log^r(n+1)\right)$ with $r>1$. Then, any aperiodic random walk with mean zero, identity as covariance and $\E[\exp\left(\log^r(\abs{S_1}+1)\right)]<\infty$ would satisfy those conditions.
\end{remark}
We now abbreviate the scales which will be used throughout the paper
\begin{align*}
\tau&=\log n,\\
T_\tau&= n/\log n.
\end{align*}

We remind the reader that $\Rcal_n$ is the number of distinct vertices the random walk has visited up to time $n$. The following limit is our main result:

\begin{theorem}
\label{thm_main}
Let $d=2$. For every $b>0$,
\begin{align*}
\lim_{n\to\infty}\frac{1}{\tau}\log P_0 \left(\Rcal_n\leq b\Tt \right)=-I(b),
\end{align*}
where
\begin{align}\label{def_ratefct}
I(b)=\inf_{\phi\in\Phi^{2\pi}(b)}\left\lbrack \frac{1}{2}\int_{\R^2} |\nabla\phi|^2(x)\mathrm{d}x \right\rbrack\, ,
\end{align}
with
\begin{align}\label{def_domainInf}
\Phi^{2\pi}(b)=\left\lbrace \phi\in H^1(\R^2)\colon \int_{\R^2} \phi^2(x) \mathrm{d}x=1,\;\int_{\R^2} (1-e^{-2\pi\phi^2(x)})\ \mathrm{d}x\leq b \right\rbrace.
\end{align}
\end{theorem}

\begin{remark}\label{remark1}
 \begin{enumerate}
 \item 
Note that the factor $\kappa$ in Theorem \ref{PPthm} is replaced by a multiple of $\pi$, analogous to the behavior of the mean in \eqref{meanEq}. The factor $2\pi$ is consequence of Lemma \ref{hardscalinglemma}, whose proof uses a computation by \cite{uchiyama2011}, involving the two-dimensional potential kernel of the random walk. 
\item 
In \cite{PP11} the author works with the simple random walk, which has covariance matrix identity times $1/d$. Hence the factor $1/d$ in Theorem \ref{PPthm}. As we work with the identity as covariance, no such factor appears in the above theorem. 
\item
The moment condition in Equation \eqref{MomentAss} is required for the proof of Proposition \ref{DV-Prop}. There, an LDP for the pair empirical measure of the random walk is established at a speed that is faster than any polynomial. This is where the decay assumption from Equation \eqref{MomentAss} is needed. In the rest of the proof, assuming the existence of finite moments is sufficient, although, in order to get a "good" approximation of the random walk density on the torus, assuming that the second moment is finite does not suffice (see the proof of Lemma \ref{BscBrwnApprxLemma}).
\end{enumerate}
\end{remark}


{
The following scaling limit of negative exponential moments is an application of Theorem \ref{thm_main} and Theorem \ref{ThmPropRateFunct} (in the Appendix). It follows from Sznitman's enlargement of obstacles method, see \cite[page 213-214]{sznitman1998brownian} (also \cite[page 362]{BBH}).
\begin{cor}
Let $c>0$, then
\begin{equation}
    \lim_{n\to\infty}\frac{1}{\tau}\log E_0\left[\ex^{-c\frac{\tau}{\Tt}\Rcal_n}\right]=-\inf_{b>0}\left[bc+I(b)\right]\, .
\end{equation}
\end{cor}
Our result also gives an easy proof\footnote{One uses Jensen's inequality to bound $E_0\left[\Rcal^k_n\right]$ from above and Theorem \ref{thm_main} with Theorem \ref{ThmPropRateFunct} for the lower bound.} of certain moment-asymptotics of $\Rcal_n$.
\begin{cor}
For $k\in (0,1)$, we have that
\begin{equation}
    E_0\left[\Rcal^k_n\right]\sim \left(\frac{n}{\log n}\right)^k\, ,
\end{equation}
where by $\sim$ we mean that there exist universal constants $C_1,C_2>0$ (only depending on $k$) such that the left-hand side is bounded from above (resp. below) by $C_1$ (resp. $C_2)$ times the right-hand side.
\end{cor}{}
}

\section{Proof of Theorem \ref{thm_main}}
The proof of Theorem \ref{thm_main} is obtained through a series of approximations. The overall structure follows the approach by \cite{BBH} and we prove several discrete analogues of their intermediate results. We begin with stating preliminary facts about the random walk on the torus in Section \ref{Compatification}. The next three sections all deal with the compactified problem, studying the range of a random walk on a finite torus. We use a coupling with the Brownian motion to apply Donsker-Varadhan theory in Section \ref{DskrVrdhSubsctn}. We then show an LDP for an approximation of the original random walk in Section \ref{LDPskeletonWalk}. 
This is the most technical part. In Section \ref{sctnApprxo} we use Talagrand's inequality to show that the error from the previous approximation is negligible on the scale of the LDP. In the last section we remove the torus restriction and finish the proof of Theorem \ref{thm_main}.

\subsection{Random walk on torus}
\label{Compatification}
We fix $N>0$ and then denote the continuum torus of length $N$ by $\L_N=[-N/2,N/2)^2$. Let $\Dt=\L_N\cap \Tt^{-1/2}\Z^2$, the rescaled periodic lattice (imposing periodic boundary conditions). Define for $x\in\R^2$
\begin{align}
\label{tranferLattice1}
& x^+=\floor{x\Ttt}\, ,\\
\label{tranferLattice2}
& x^-=\floor{x\Ttt}\Tt^{-1/2}\, .
\end{align}
Suppose that $X_i$ are i.i.d random variables on $\Z^2$ with all moments being finite.

Let $P_x,E_x$ denote the measure and expectation of the planar random walk on $\Z^2$ started at $x$, given by the sum of the $X_i$'s. For $x\in\L_N$, we write $\Pn_x,\En_x$ the measure of that random walk projected onto $\Dt$, where we always implicitly use rounding: $\Pn_x$ is to be understood as $\Pn_{x^{-}}\circ \texttt{M}^{-1}$, where $\texttt{M}$ is the map on the path space sending each space time point $(x,t)\mapsto (x^{-},t)$ and analogous for $\En_x$. The random walk under $\Pn_x$ is rescaled in space and converges to the Brownian motion if we stretch time by $\Tt$.

For $x,y\in \Z^2$, let $p_t(x,y)$ be the transition kernel from the point $x$ to the point $y$ in time $t$, associated to $P_x$. For $x,y\in\Dt$, we write $\pn_t(x,y)$ for the kernel associated to $\Pn_x$. Moreover, denote $(S_t)_{t\ge 0}$ the family of coordinate projections onto the respective space (where we interpret $S_t=S_{\floor{t}}$ for $t\notin \N$). \\
Note that the transition probability of random walk on torus is then given by an infinite sum of transition probabilities of random walk on $\Z^2$, i.e.\ for $x,y\in\Dt$
\begin{align}\label{tranPDeltaN}
\pn_t(x,y)
=&\Pn_{x}(S_t=y)=\sum_{z\in \Z^2}P_{x^+}(S_t=y^+ +{z^+}N) \nonumber\\
=&\sum_{z\in \Z^2}p_t(x^+,y^++{z^+}N)\, .
\end{align}
Along the path $x\mapsto x^-$, above formula of the transition probability can be extended for $x,y\in \L_N$ by $\pn_t(x,y)=\pn_t(x^-,y^-)$. 

We denote the measure of the Brownian paths on $\R^2$ started at $x$ by $\P_x$ and $\P_x^{(N)}$ for the Brownian motion projected onto $\L_N$.
The Brownian transition kernel from the point $x$ to the point $y$ at time $t$ is given by
$$
\p_t(x,y)=\frac{1}{2\pi t}\exp\left( -\frac{\abs{y-x}^2}{2t} \right).
$$ 
Similarly for $x,y\in \L_N$, denote the Brownian transition kernel on the torus $\L_N$ by 
\begin{align*}
\p^{(N)}_t(x,y)
= \sum_{z\in\Z^2}\p_t(x,y+Nz)\, .
\end{align*}
Note that both $\pn_t(x,y)$ and $\p^{(N)}_t(x,y)$ again depend on $N,n$. Furthermore, by the local central limit Theorem \cite[Theorem 2.1.1]{Lawler2010}, one can approximate the transition probability with the Brownian kernel. For $x,y\in\Z^
2$,
\begin{equation}\label{LCLT}
\begin{split}
p_{t}(x,y)
=& \p_t(x,y) + A_t(x,y)\\
=& \frac{1}{2\pi t}\exp\left(-\frac{|x-y|^2}{2t}\right) + A_t(x,y)\, ,
\end{split}
\end{equation}
where
\begin{equation*}
A_t(x,y)=\min\Big\lbrace \Ocal(t^{-2}),\Ocal\left(|x-y|^{-2}t^{-1}\right)\Big\rbrace\, .
\end{equation*}
On the rescaled lattice $\Dt$, a similar result carries over for $\pn_t(x,y)$ with $x,y\in\L_N$, see Lemma \ref{BscBrwnApprxLemma}.
\begin{remark}
Note that whether $(S_t)_t$ lives on $\Z^2$ or the rescaled torus $\Dt$ is indicated by the reference measure. If it is $\Pn_x$, then the random walk lives on the rescaled lattice, otherwise on $\Z^2$. We do not use the notation $S^{-}_t$ for that reason.
\end{remark}

\subsection{Donsker-Varadhan LDP}
\label{DskrVrdhSubsctn}
In this subsection, we introduce an empirical functional based on the random walk and show a large deviation principle for it. The key will be the so called Donsker-Varadhan theory (see \cite{DV75}). 

Define the rescaled empirical measure $L_{n,\e}$ on $\Mcal_1(\L_N\times\L_N)$, the space of probability measures on $(\L_N\times\L_N$), by
\begin{equation}\label{empP}
    L_{n,\e}=\frac{\e }{\tau}\sum_{i=1}^{ \tau/\e}\delta_{\left(S_{(i-1)\e \Tt},S_{i\e \Tt} \right)}\, ,
\end{equation}
where the random walk $(S_i)_i$ lives on the rescaled torus $\Dt$. Let $I_\e^{(2)}\colon \Mcal_1(\L_N\times \L_N)\to [0,\infty)$ be the entropy function which is defined in the following way
\begin{equation}\label{entropyFNCT}
    I_\e^{(2)}(\mu)=\begin{cases}
     h(\mu|\mu_1\otimes\p^{(N)}_\e) &\text{ if }\mu_1=\mu_2\, ,\\
    +\infty &\text{ otherwise}\, ,
    \end{cases}
\end{equation}
where $\p^{(N)}_\e=\p^{(N)}_\e(y-x)\d y$ is the measure induced by the Brownian transition kernel($x$ is with respect to $\mu_1$). Furthermore, $h$ denotes the usual entropy between two measures and $\mu_i$ is the $i$-th marginal of $\mu$ for $i=1,2$.

\begin{proposition}\label{DV-Prop}
Under $\Pn_x$ the empirical functional $L_{n,\e}$ satisfies an LDP with speed $\tau$ and good rate function $\e^{-1}I_{\e}^{(2)}$ (in the weak topology), where $I^{(2)}_\e$ was defined in Equation \eqref{entropyFNCT}.
\end{proposition}

\begin{proof}
The proof consists of two steps. Firstly, we remind the reader of exponential equivalence for empirical measures in the context of large deviation theory, and then prove that $L_{n,\e}$ is exponentially equivalent to a different empirical functional $\widehat{L_{n,\e}}$ on the right scale. In the second step, we show that $\widehat{L_{n,\e}}$ satisfies a large deviation principle with good rate function from which the proposition then follows.\\
Recall that the Prokhorov distance $\d_P$ between two probability measures $\mu,\nu$ on some metric space $(E,\d)$ (with $\sigma-$algebra $\Ecal$) is defined as
\begin{equation*}
    \d_P(\mu,\nu)=\inf\{\e>0\,\, \forall B\in \Ecal\colon \, \mu(B)\le \nu(B^\e)+\e\text{  and  }\nu(B)\le \mu(B^\e)+\e\}\, ,
\end{equation*}
where $B^\e=\{x\in E\colon \d(x,B)<\e\}$. Furthermore, let $(M_t)_{t\ge 0}$ be a standard Brownian motion on $\L_N$. We define the empirical functional $\widehat{L_{n,\e}}$ in the following way
\begin{equation*}
   \widehat{L_{n,\e}}=\frac{\e}{\tau}\sum_{i=1}^{\tau/\e}\delta_{(M_{(i-1)\e},M_{i\e})}\, . 
\end{equation*}
We now want to show that $\widehat{L_{n,\e}}$ and $L_{n,\e}$ are exponentially equivalent, i.e.\ there exists a common probability space $(\Omega,\Acal)$ and a probability measure $\Pr$ on $\Omega$ such that the random variable $(L_{n,\e},\widehat{L_{n,\e}})$ has the right marginal law and that for every $\delta>0$, we have that
\begin{equation*}
    \lim_{n\to\infty}\frac{1}{\tau}\log \Pr(\Gamma_\delta)=-\infty\, ,
\end{equation*}
with $\Gamma_\delta=\{(\mu,\nu)\in \Mcal_1(\L_N\times\L_N)^2\colon\, \d_P(\mu,\nu)>\delta\}$. By \cite[Theorem 4]{einmahl1989extensions} there exists a coupling $\P'$ between the $\P_x^{(N)}$ (the Brownian motion) and  $\Pn_x$ (the random walk on $\Dt$) satisfying
\begin{equation*}
    \P'(A):=\P'\left(\sup_{i\le \tau/\e}\abs{M_{i\e}-S_{i\e \Tt}}\ge \delta\right)\le \frac{n}{H\left(c\delta \sqrt{n}\right)}\, ,
\end{equation*}
for some fixed constant $c>0$ (depending on the distribution of $X_1$). Choosing $\Pr=\P'$ now implies the exponential equivalence, as
\begin{equation*}
    \Pr(\Gamma_\delta)\le \Pr(A)+\Pr(\Gamma_\delta\cap A^c)\, .
\end{equation*}
Indeed, we have
\begin{equation*}
    \Pr(\Gamma_\delta\cap A^c)=0\, .
\end{equation*}
Thus
\begin{equation}
    \lim_{n\to \infty}\frac{1}{\tau}\log \Pr(\Gamma_\delta)\le\lim_{n\to \infty}\frac{1}{\tau}\log \Pr(A)=-\infty\, ,
\end{equation}
by the assumption on $H$. This shows that $L_{n,\e}$ is exponentially equivalent to $\widehat{L_{n,\e}}$.\\
It remains to show that $\widehat{L_{n,\e}}$ satisfies an LDP with good rate function $\e^{-1}I_{\e}^{(2)}$. By \cite[Theorem 1.5]{BOLTHAUSEN87}, this is indeed the case. Implied by \cite[Lemma 6.2.12]{dembo2009large}, the rate function is good in the weak topology, so by \cite[Theorem 4.2.13]{dembo2009large}  the result now follows by exponential approximation.
\end{proof}

\subsection{An LDP for the skeleton walk}\label{LDPskeletonWalk}
In this section, we prove an LDP for the skeleton walk (defined below). This will make use of subsection \ref{DskrVrdhSubsctn} as well as some random walk estimates.\\
Define the skeleton walk
\begin{align*}
\Scal_{n,\eps}=\{S_{i\eps \Tt}\}_{1\le i\le \tau/\eps},
\end{align*}
where $(S_t)_t$ is distributed under $\Pn_0$ (and thus lives on on $\Dt$). Furthermore, define $\Pne,\Ene$ the conditional law/expectation given $\Scal_{n,\e}$. The aim will be to write the range of the skeleton walk as a functional of the empirical measure introduced in subsection \ref{DskrVrdhSubsctn} and then use the contraction principle. Several approximations steps will be necessary.\\
For $\eta>0$, let $\Phi_\eta\colon \Mcal_1(\L_N\times \L_N)\to [0,\infty)$ be defined as
\begin{equation*}
    \Phi_\eta(\mu)=\int_{\L_N}\d x\left( 1-\exp\left[ -\eta  \int_{\L_N\times \L_N} 2\pi\phi_\e(y-x,z-x)\mu(\d y,\d z)\right]\right)\, ,
\end{equation*}
with
\begin{equation}\label{defPhie}
    \phi_\e(y,z)=\frac{\int_0^\e \d s\, \p_{s}^{(N)}(-y)\p_{\e-s}^{(N)}(z)}{\p_{\e}^{(N)}(z-y)}\, .
\end{equation}
The key result of this section is the following LDP for the measure conditioned on the skeleton walk:
\begin{proposition}\label{sceletonLDP}
$\Ene\left[\frac{1}{\Tt}\Rcal_n\right]$ satisfies an LDP on $\R^+$ with speed $\tau$ and rate function
\begin{equation*}
    J_{\e}(b)=\inf\big\{\e^{-1}I_{\e}^{(2)}(\mu)\colon \mu\in\Mcal_1(\L_N\times \L_N),\Phi_{1/\e}(\mu)=b\big\}\, .
\end{equation*}
\end{proposition}

\begin{proof}
\textbf{Step 1:}
Unless stated otherwise, in this proof $(S_t)_t$ is distributed with respect to $\Pn_0$ and therefore takes values in $\Dt$. Define for $1\le i\le \tau/\e$ the set 
$$
\Wcal_i=\left\lbrace S_j\colon (i-1)\e\Tt \leq j \leq i\e\Tt \right\rbrace\, .
$$
We begin by cutting holes into our range, to lessen dependence. Let us denote by
$$
Q_\tau:=\sqrt{\log \Tt\log\log \Tt}^{-1}\, ,
$$ 
and define
\begin{equation*}
\begin{split}
\Wcal_i^\tau=\Wcal_i\setminus\!\big\lbrace & S_j\colon \abs{S_j-S_{(i-1)\e \Tt}}<Q_\tau \\
&
\text{ or } \abs{S_j-S_{(i-1)\e \Tt}}<Q_\tau;\, (i-1)\e\Tt\!<j<\!i\e \Tt\big\rbrace.
\end{split}
\end{equation*}
Meanwhile, set
$$
\Rcal_n^\tau=\#\Big\{\bigcup_{i=1}^{\tau/\e}\Wcal_i^\tau\Big\}\,.
$$
Notice that the cutting holes procedure corresponds to removing balls of size $Q_\tau$ on $\L_N$. Then it follows for the random walk on $\Dt$ with some constant $c_1>0$
\begin{equation}\label{holecutbound}
\frac{1}{\Tt}\abs{\Rcal_n-\Rcal_n^\tau}\leq 
\frac{1}{\Tt}c_1\frac{\tau}{\eps} \left(2 Q_\tau\sqrt{\Tt}\right)^2
\to 0\, ,
\end{equation}
 as $n\to\infty$, since we cut at most $\tau/\epsilon$ many balls, each of which has radius $Q_\tau$ (and thus contains approximately $\left( Q_\tau\sqrt{\Tt}\right)^2$ points). \\
\textbf{Step 2:} Let us define the stopping time $\sigma=\min\{n\ge 0\colon S_n=0\}$. Then, for $y,z\in \L_N$, the bridge kernel is denoted by
\begin{equation}\label{bridgeP}
\begin{split}
b_{n,\e}(y,z)
:=& \Pn_{y}\left(\sigma\le \e\Tt\mid S_{\e \Tt}=z\right)\\
=& \Pn_{y^-}\left(\sigma\le \e\Tt\mid S_{\e \Tt}=z^-\right)\\
=&\Pn_{y,z}\left(\sigma\le \e\Tt\right)\, .
\end{split}
\end{equation}
We denoted the bridge probability measure by $\Pn_{y,z}\left(\,\cdot\,\right):= \Pn_{y^-}\left(\,\cdot\,\mid S_{\e \Tt}=z^-\right)$. We now have
\begin{align}\label{eqEmpMrsRplc}
\Ene &\left[\frac{1}{\Tt}\Rcal^\tau_n\right]
=\frac{1}{\Tt}\sum_{x\in\Dt} \left(1-\Pne\Big(x\notin\bigcup_{i=1}^{\tau/\e}\Wcal_i^\tau\Big)\right) \nonumber\\
&=\int_{\L_N}\d x \left(1-\Pne\Big(x^-\notin\bigcup_{i=1}^{\tau/\e}\Wcal_i^\tau\Big)\right)\\
&=\int_{\L_N}\d x \left(1-\exp\Big(\sum_{i=1}^{\tau/\e}\log\left[1-\Pn_{n,\e}\left(x^-\notin\Wcal_i^\tau\right)\right]\Big)\right) \nonumber\\
&=\int_{\L_N}\d x \left(1-\exp\left(\frac{\tau}{\e}\int_{\L_N\times \L_N}\!\!\!\!\!\!\!\!\!\! L_{n,\e}(\d y, \d z)\log\left[1-b_{n,\e}^{Q_\tau}(y-x,z-x)\right]\right)\right)\, ,\nonumber
\end{align}
where in the last equality we inserted the definition of empirical measure given in \eqref{empP} and $b_{n,\e}^{Q_\tau}(y,z)=b_{n,\e}(y,z)\1\{y,z\notin B_{Q^\tau}\}$. We remind the reader that in general $B_{\rho}$ denotes the centered ball of radius $\rho$ in $\R^2$.

\hspace{-4mm}\textbf{Step 3:} We begin with an important proposition.
\begin{proposition}\label{PropHard}
For $y, z\in\L_N$, we have
\begin{enumerate}[(a)]
\item $\lim\limits_{\tau\to\infty}\sup\limits_{y,z\notin B_{Q_\tau}}b_{n,\eps}(y,z)=0.$
\item $\lim\limits_{\tau\to\infty}\sup\limits_{y,z\notin B_{\rho}}\babs{ \tau b_{n,\eps}(y,z)-2\pi \phi_\eps(y,z)}=0$, for all $0< \rho<  N/4$.
\end{enumerate}
\end{proposition}
We defer the proof of this proposition to the end of the section, as it is lengthy and might distract the reader from the overall structure of the proof.

Let us now introduce two new functions, 
$\Phi_{n,\eta,\rho}\colon \Mcal_1(\L_N\times\L_N)\to [0,\infty)$
\begin{equation*}
    \Phi_{n,\eta,\rho}(\mu)=\int\d x\left(1-\exp\left(-\eta \tau\int_{\L_N\times\L_N}b_{n,\e}^{\rho}(y-x,z-x)\d \mu(x,y)\right)\right)\, ,
\end{equation*}
and $\Phi_{\infty,\eta,\rho}\colon \Mcal_1(\L_N\times\L_N)\to [0,\infty)$
\begin{equation*}
    \Phi_{\infty,\eta,\rho}(\mu)=\int\d x\left(1-\exp\left(- \eta\int_{\L_N\times\L_N} 2\pi\phi_{\e}^{\rho}(y-x,z-x)\d \mu(x,y)\right)\right)\,  ,
\end{equation*}
with $\phi_\e^\rho(y,z)=\phi_\e(y,z)\1\{y,z\notin B_{\rho}\}$.
For a sequence $\delta_\tau>0$ with $\lim_{\tau\to\infty}\delta_\tau=0$, we have by the Proposition \ref{PropHard} and \eqref{eqEmpMrsRplc}
\begin{equation}\label{ERntauBound}
    \Phi_{n,1/\e,Q_\tau}(L_{n,\e})\le \Ene \left[\Tt^{-1}\Rcal_N^\tau \right]\le \Phi_{n,(1+\delta_\tau)/\e,Q_\tau}(L_{n,\e})\, ,
\end{equation}
The next lemma summarizes continuity properties of the above defined functions.

\begin{lemma}\label{ContinuityProp}
Fix $N>0$. There exist absolute constants $C_i$ such that
\begin{enumerate}[(a)]
\item $\babs{\Phi_{n,\eta,\rho}(\mu)-\Phi_{n,\eta,\rho'}(\mu)}\le C_1\eta\left(\rho^2+\rho'^2\right)$ for all $\eta,\mu$ and $n\geq n_0(\rho, \rho')$.
\item $\babs{\Phi_{n,\eta,\rho}(\mu)-\Phi_{n,\eta',\rho}(\mu)}\le C_2|\eta-\eta'|$ for all $\rho,\mu$ and $n\geq n_0(\rho)$.
\item $\babs{\Phi_{n,\eta,\rho}(\mu)-\Phi_{\infty,\eta,\rho}(\mu)}\le C_3 \eta o_\rho(1)$ for all $\eta,\mu$.
\item $\babs{\Phi_{\infty,\eta,\rho}(\mu)-\Phi_{\infty,\eta,0}(\mu)}\le C_4\eta \rho^2$ for all $\eta,\mu$.
\item $\babs{\Phi_{\infty,\eta,0}(\mu)-\Phi_{\infty,\eta,0}(\mu')}\le C_5\eta\norm{\mu-\mu'}_{tv}$ for all $\eta$ and $\norm{\cdot}_{tv}$ denotes the total variation norm.
\end{enumerate}
\end{lemma}
\begin{proof}
Using the definition of $\phi_\e(\cdot,\cdot)$ in Equation \eqref{defPhie}, we notice that
\begin{align*}
&  \int_{\L_N}\d x \int_{(\L_N\times B_\rho(x))\cup (B_\rho(x)\times \L_N)}\mu(\d y,\d z) \phi_\e(y-x,z-x)\\
&\leq \int_{\L_N\times \L_N}\mu(\d y,\d z) \int_{B_\rho(y)}\d x \int_{0}^{\e}\d s\frac{\p^{(N)}_{\e-s}(y,x)\p^{(N)}_{s}(x,z)}{\p^{(N)}_{\e}(y,z)}\\
&+ \int_{\L_N \times \L_N}\mu(\d y,\d z) \int_{B_\rho(z)}\d x 
\int_{0}^{\e}\d s\frac{\p^{(N)}_{\e-s}(y,x)\p^{(N)}_{s}(x,z)}{\p^{(N)}_{\e}(y,z)}\,,
\end{align*}
where the last inequality followed by exchanging the order of integration. We can estimate both parts in the same way, so it suffices to bound
\begin{equation*}
    \int_{\L_N\times\L_N} \mu(\d y,\d z)\int_{B_\rho(y)}\d x \int_{0}^{\e}\d s\frac{\p^{(N)}_{\e-s}(y,x)\p^{(N)}_{s}(x,z)}{\p^{(N)}_{\e}(y,z)}\, .
\end{equation*}
As $N$ is fixed, $\p^{(N)}_{\e}(y,z)$ and $\p^{(N)}_{\e-s}(y,x)$ can be bounded from above and below by non-zero constants, in particular uniformly in $z,y,x$. Thus, it remains to verify
\begin{equation*}
    \int_0^{\e}\d s\int_{B_\rho(0)}\d x\, \p^{(N)}_{s}(x)=\Ocal(\rho^2)\, ,
\end{equation*}
which follows easily after a change of variables. 
Therefore, we get
\begin{equation}\label{EstPhie}
\int_{\L_N}\d x \int_{\L_N\times \L_N}\mu(\d y,\d z) \big|\phi_\e(y-x,z-x)-\phi_\e^\rho(y-x,z-x)\big|
=\Ocal(\rho^2)\,.
\end{equation}
\begin{enumerate}[(a)]
\item Keeping the above bound in mind, we have
\begin{align*}
&\babs{\Phi_{n,\eta,\rho}(\mu)-\Phi_{n,\eta,\rho'}(\mu)}\\
 &\leq\eta\int_{\L_N}\d x\int_{\L_N\times \L_N}\mu(\d y,\d z)\left| \tau b_{n,\e}^{\rho}(y-x,z-x)
-\tau b_{n,\e}^{\rho'}(y-x,z-x)\right|\\
 &=\eta\int_{\L_N}\d x\int_{\L_N\times \L_N}\mu(\d y,\d z)\\
&\qquad\qquad\qquad\times \left( \left| 2\pi\phi_{\e}^{\rho}(y-x,z-x)
-2\pi\phi_{\e}^{\rho'}(y-x,z-x)\right| +o_{\rho,\rho'}(1)\right)\\
 & \leq C\eta \left( (\rho^2+\rho'^2)+|\L_N|o_{\rho,\rho'}(1)\right)\,,
\end{align*}
where we used the part $(b)$ from Proposition \ref{PropHard}.

\item Similarly, it holds
\begin{align*}
&\babs{\Phi_{n,\eta,\rho}(\mu)-\Phi_{n,\eta',\rho}(\mu)} \\
&\leq |\eta-\eta'| \int_{\L_N}\d x\int_{\L_N\times \L_N}\mu(\d y,\d z)\left| \tau b_{n,\e}^{\rho}(y-x,z-x)\right|\\
&= |\eta-\eta'| \int_{\L_N}\d x\int_{\L_N\times \L_N}\mu(\d y,\d z)\big( \left| 2\pi \phi_{\e}^{\rho}(y-x,z-x)\right|+o_{\rho}(1) \big)\\
&\leq C\left( |\eta-\eta'|+|\L_N|o_{\rho}(1)\right)\,.
\end{align*}

\item We obtain
\begin{align*}
&\babs{\Phi_{n,\eta,\rho}(\mu)-\Phi_{\infty,\eta,\rho}(\mu)}\\
&\leq \eta\int_{\L_N}\d x\int_{\L_N\times \L_N}\mu(\d y,\d z)\left| \tau b_{n,\e}^{\rho}(y-x,z-x)
-2\pi\phi_{\e}^{\rho}(y-x,z-x)\right|\\
&\le C \eta|\L_N| o_\rho(1)\,.
\end{align*}

\item Employing \eqref{EstPhie} directly, we obtain
\begin{align*}
&\babs{\Phi_{\infty,\eta,\rho}(\mu)-\Phi_{\infty,\eta,0}(\mu)}\\
&\leq \eta \int_{\L_N}\d x\int_{\L_N\times \L_N}\mu(\d y,\d z) \big|2\pi \phi_\e (y-x,z-x)-2\pi \phi_\e^\rho (y-x,z-x) \big|\\
&\le C_4\eta \rho^2\,.
\end{align*}

\item In the same way, we get
\begin{align*}
&\babs{\Phi_{\infty,\eta,0}(\mu)-\Phi_{\infty,\eta,0}(\mu')}\\
 &\leq\eta \int_{\L_N}\d x\int_{\L_N\times \L_N}|\mu-\mu'| (\d y,\d z) 2\pi \phi_\e^\rho (y-x,z-x) \\
&\le C_5\eta\norm{\mu-\mu'}_{tv}\,.
\end{align*}
\end{enumerate}
This concludes the proof of Lemma \ref{ContinuityProp}.\qed
\end{proof}

\hspace{-4mm}\textbf{Step 4:} 
We approximate $\Ene[\Tt^{-1}\Rcal_n]$ by $\Ene[\Tt^{-1}\Rcal_n^\tau]$, with the error given in \eqref{holecutbound}. Then we combine \eqref{ERntauBound}, Proposition \ref{PropHard} with the previous lemma to get 
\begin{align*}
&\Ene\left[ \Tt^{-1}\Rcal_n \right]\\
\leq& \Ene\left[ \Tt^{-1}\Rcal_n^\tau \right]+ c_1 \frac{\tau Q_\tau^2}{\e} \\
\leq& \Phi_{n,(1+\delta_\tau)/\e,Q_\tau}\left(L_{n,\e}\right) + c_1  \frac{\tau Q_\tau^2}{\e} \\
\leq& \Phi_{n,1/\e,\rho}\left(L_{n,\e}\right) + c_1  \frac{\tau Q_\tau^2}{\e} + C_1(Q_\tau^2+\rho^2) + C_2\frac{\delta_\tau}{\e} \\
\leq& \Phi_{\infty,1/\e,0}\left(L_{n,\e}\right) + c_1  \frac{\tau Q_\tau^2}{\e} + C_1(Q_\tau^2+\rho^2) + C_2\frac{\delta_\tau}{\e} + C_3\frac{o_\rho(1)}{\e} + C_4 \frac{\rho^2}{\e}\,.
\end{align*}
Therefore, we arrive at
\begin{align*}
&\left\| \Ene\left[ \Tt^{-1}\Rcal_n \right] -\Phi_{\infty,1/\e,0}\left( L_{n,\e} \right)\right\| _\infty  \\
&\leq c_1 \frac{\tau Q_\tau^2}{\e} + C_1(Q_\tau^2+\rho^2) + C_2\frac{\delta_\tau}{\e} + C_3\frac{o_\rho(1)}{\e} + C_4\frac{\rho^2}{\e}\, .
\end{align*}
Note $\lim_{\tau\to\infty}\delta_\tau=0$. By letting $\tau\to\infty$, $\rho$ to zero, this implies for all $\e >0$
\begin{equation*}
    \lim_{\tau\to \infty}\norm{\Ene[\Tt^{-1}\Rcal_n]-\Phi_{\infty,1/\e,0}(L_{n,\e})}_\infty=0\, .
\end{equation*}
Using the contraction principle and Proposition \ref{DV-Prop}, we conclude the proof of Proposition \ref{sceletonLDP}.

\end{proof}

\subsubsection{Proof of Proposition \ref{PropHard}}
\begin{proof}
\begin{enumerate}[(a)]
\item Note by \eqref{bridgeP} that $b_{n,\e}(a,b)=\Pn_{a,b}\left(\sigma\le \e\Tt\right)$, where $\sigma=\min\{n\colon S_n=0\}$.
We first remove the bridge and then the torus restriction. Since
\begin{equation}
    \sup_{0\le t\le \e\Tt/2}\pn_{\e\Tt-t}(0)/\pn_{\e\Tt}(0)<\infty\, ,
\end{equation}
there exists a constant $c_2>0$ such that for $a,b\in\L_N$ and $0< Q_\tau<N/4$,
\begin{equation*}
    \sup_{a,b\notin B_{Q_\tau}}b_{n,\e}(a,b)\le 2\, c_2\sup_{a\notin B_{Q_\tau}}\Pn_a\left(\sigma\le \frac{\e\Tt}{2}\right)\, .
\end{equation*}
Let $\sigma_r$ be the hitting time of the boundary of the centered ball with radius $r>0$ in $\R^2$, i.e.\ $\sigma_r=\min\{n\ge 0\colon S_n\in B_{N/r}\}$. We decompose the right-hand side of above equation into
\begin{align}\label{Pstoppingdecomp}
  &\Pn_a\left(\sigma\le \frac{\e\Tt}{2}\right)\nonumber\\
  =& \Pn_a\left(\sigma\le \frac{\e\Tt}{2},\, \sigma<\sigma_{N/2}\right)+\Pn_a\left(\sigma\le \frac{\e\Tt}{2},\, \sigma\geq \sigma_{N/2}\right) \, ,
\end{align}
and it holds that
\begin{equation*}
  \Pn_a\left(\sigma\le \frac{\e\Tt}{2},\, \sigma \geq \sigma_{N/2}\right)\le c_3\sup_{x\in \partial B_{Q_\tau}}\Pn_{x}\left(\sigma\le \frac{\e \Tt}{2}\right)\, ,
\end{equation*}
where we define
\begin{equation*}
  c_3:=  \left(\sup_n\sup_{x\in \partial B_{N/2}}\Pn_{x}\left(\sigma_{N/4}<\frac{\e \Tt}{2}\right)\right)\,.
\end{equation*}
Indeed, the random walk must hit $\partial B_{N/2}$ before it hits the origin, therefore it must pass through $\partial B_{N/4}$ on its way. In particular, by Donsker's invariance principle, we have
\begin{equation*}
\lim_{n\to\infty}\sup_{x\in \partial B_{N/2}}\Pn_{x}\left(\sigma_{N/4}<\frac{\e \Tt}{2}\right)
= \sup_{x\in \partial B_{N/2}}\P^{(N)}_{x}\left(\sigma_{N/4}<\frac{\e }{2}\right)<1\, ,
\end{equation*}
where we were allowed to exchange the limit and supremum as $x\mapsto\P^{(N)}_{x}\left(\sigma_{N/4}<\frac{\e }{2}\right)$ is constant in $|x|$.
And thus, $c_3<1$. Therefore, \eqref{Pstoppingdecomp} has been reduced to
\begin{equation*}
    \sup_{a\notin B_{Q_\tau}}\Pn_{a}\left(\sigma\le \frac{\e\Tt}{2}\right)\le \frac{1}{1-c_3}\sup_{a\notin B_{Q_\tau}}\Pn_{a}\left(\sigma\le \frac{\e\Tt}{2},\, \sigma<\sigma_{N/2}\right)\, .
\end{equation*}
On the event $\sigma<\sigma_{N/2}$ (i.e. the random walk has not hit $\partial B_{N/2}$ yet), the random walk on $\Dt$ behaves as a random walk on $\Z^2$, and thus we can bound
\begin{equation*}
\Pn_{a}\left(\sigma\le \frac{\e\Tt}{2},\, \sigma<\sigma_{N/2}\right)\le P_{a^+}\left(\sigma\le \frac{\e\Tt}{2}\right)\, .
\end{equation*}
Note that we needed to undo the scaling of the factor $\Tt^{1/2}$ when we changed from $\Dt$ to $\Z^2$. Putting the previous steps together
\begin{equation*}
    \sup_{a,b\notin B_{Q_\tau}}b_{n,\e}(a,b)
\le \frac{c_2}{1-c_3}\sup_{a\notin B_{Q_\tau}} P_{a^+}\left(\sigma\le \frac{\e\Tt}{2}\right)\, .
\end{equation*}
We decompose now
\begin{equation*}
    P_{a^+}\left(\sigma\le \frac{\e\Tt}{2}\right)\le P_{a^+}\left(\sigma\le \omega\right)+P_{a^+}\left(\omega\le \frac{\e\Tt}{2}\right)\, ,
\end{equation*}
where $\omega$ is the first entrance time into the complement of the centered ball with radius $R\Ttt$. Applying the estimate given in \cite[Proposition 6.4.3]{Lawler2010},  we have that 
\begin{align*}
 & P_{a^+}\left(\sigma\le \omega\right)\\
=& \left(1-\frac{\log (|a|\sqrt{\Tt})+\Ocal \left(|a\sqrt{\Tt}|^{-1}\right)}{\log \left(R\sqrt{\Tt}\right)}\right)\left(1+\Ocal\left(\frac{1}{\log \left(R\sqrt{\Tt}\right)}\right)\right)\, ,
\end{align*}
Furthermore, by bounds on the maximum in \cite[Equation 2.6]{Lawler2010} we have that for some $c>0$
\begin{equation*}
    P_{a^+}\left(\omega\le \frac{\e\Tt}{2}\right)\le 2 c\left(\frac{R-|a|}{\sqrt{\e/2}}\right)^{-4} \, .
\end{equation*}
Setting $R=|a|+Q_\tau^{-1/4}$ and using $\log(1+x)\le x$ allow us to bound for some $C>0$
\begin{equation*}
P_{a^+}\left(\sigma\le \frac{\e\Tt}{2}\right)\le 2\e^2 c Q_\tau +C\frac{\left(\log \Tt\log\log \Tt\right)^{3/8}}{1/2\log\Tt+\log Q_\tau}
\overset{\tau\to\infty}{\longrightarrow}0 \, .
\end{equation*}
\item Let $0<\delta<\e/2$ and for $a,b\in \L_N$, we define
\begin{equation}\label{Defbne3}
b_{n,\e}(a,b,\delta\Tt)
=\Pn_{a,b}\left( \delta \Tt<\sigma <(\e-\delta)\Tt \right)\, .
\end{equation}
We claim that
\begin{equation}\label{reducingProofb}
\lim_{\delta\downarrow 0}\lim_{n\to \infty}\sup_{a,b\notin B_{\rho}}\babs{ \tau b_{n,\e}(a,b)-\tau b_{n,\e}(a,b,\delta \Tt)}=0\, .
\end{equation}
Notice that
\begin{align*}
& \babs{ \tau b_{n,\e}(a,b)-\tau b_{n,\e}(a,b,\delta \Tt)}\\
&\leq \tau \left| \Pn_{a,b}\left(\sigma\le \e\Tt\right) - \Pn_{a,b}\left( \delta \Tt<\sigma <(\e-\delta)\Tt \right)\right|\\
&\leq \tau \left| \Pn_{a,b}\left(\sigma\le \delta
\Tt\right)\right|+\tau \left| \Pn_{b,a}\left(\sigma\le \delta
\Tt\right)\right|\,,
\end{align*}
and repeat the argument of part $(a)$ to obtain
\begin{align*}
\sup_{a,b\notin B_{\rho}} \Pn_{a,b}\left(\sigma\le \delta\Tt\right)
\leq c_4\sup_{y\notin B_{\rho}}P_{y^+}\left(\sigma\le \delta\Tt\right)
\end{align*}
with some constant $c_4>0$. Therefore, in order to prove \eqref{reducingProofb} it suffices to show for $y\in \partial B_{\rho}$
\begin{equation*}
    \lim_{\delta \downarrow 0}\lim_{\tau\to\infty}\tau P_{{y}^+}(\sigma<\delta \Tt)=0\,, 
\end{equation*}
However, as it was shown in \cite[Theorem 1.6]{uchiyama2011}
\begin{equation*}
    \lim_{\tau\to\infty}\tau P_{{y}^+}(\sigma<\delta \Tt)=\int_{\rho^2/\delta}^\infty\frac{\ex^{-u}}{u}\d u=2\pi \int_0^\delta \p^{(N)}_u(y)\d u.
\end{equation*}
Therefore, it suffices to show
\begin{equation}
    \lim_{\delta\downarrow 0}\lim_{\tau\to \infty}\babs{\tau b_{n,\e}(a,b,\delta \Tt)-2\pi \phi(a,b)}=0\, .
\end{equation}
We expand $b_{n,\e}(\cdot,\cdot,\cdot)$ 
\begin{equation*}
    b_{n,\e}(a,b,\delta \Tt)=\!\sum_{k=\delta \Tt}^{(\e-\delta)\Tt}\!\frac{\Pn_{a}(S_k=0,S_1,\ldots,S_{k-1}\neq 0)\Pn_0(S_{\e\Tt-k }=b)}{\Pn_{a}(S_{\e \Tt}=b)}\, .
\end{equation*}
We rewrite this using discrete integration by parts
\begin{align*}
&b_{n,\e}(a,b,\delta\Tt)\\
&= \frac{1}{\pn_{\e\Tt}\left(a-b\right)}\int_{\delta}^{\e-\delta}\d s \, \Pn_{a}(\sigma<{s\Tt})\frac{\left( \pn_{(\e-\delta-s)\Tt+1}(b)-\pn_{(\e-\delta-s)\Tt}(b)\right)}{\Tt^{-1}}\, .
\end{align*}
To handle the scaling limit, we introduce another lemma.
\begin{lemma}\label{hardscalinglemma}
We have that for $s>0$ and $a\in\L_N\setminus\{0\}$
\begin{equation}
    \lim_{n\to \infty}\tau \Pn_{a}(\sigma<{s\Tt})={2\pi}\int_0^{s}\p^{(N)}_t(a)\,\d t\, .
\end{equation}
\end{lemma}
\begin{proof}
Let $H_z$ be the hitting time of the point $z\in \R^2$.  Applying the definition of our discretized torus, by combining inclusion exclusion (as we are on the torus) and \cite[Equation 2.6]{Lawler2010}, we obtain
\begin{equation*}
\begin{split}
    &\Pn_{a}(\sigma<{s\Tt})=\sum_{k=1}^{\sqrt{\log \Tt}}(-1)^{k+1}\\
    &\qquad\qquad\quad\times\sum_{\genfrac{}{}{0pt}{2}{\abs{z_i}<\sqrt{\log \Tt}}{i=1,\ldots k}}P_0\left(H_{a^++Nz_i^+}<{s\Tt},\,\forall i=1,\ldots k\right)+o\left(\tau^{-1}\right)\, ,
\end{split}
\end{equation*}
where the summation takes place in a way that we choose an arbitrary ordering $\{z_1,\ldots,z_M\}$, whose $z_i$ with  $|z_i|\le \sqrt{\log \Tt}$ (where $M$ is the number of such $z_i$'s). The sum runs over all $z_{i_1},\ldots,z_{i_k}$ with  $1\le i_1<\ldots<i_k\le M$. Recall the fact that the random walk under $P_0$ lives on $\Z^d$ and the $z_i$'s are all distinct. Note that by \cite[Theorem 1.6]{uchiyama2011} the terms with $k=1$ converge to the desired scaling limit, by the subsequent argument for the error terms in Equation \eqref{errorTerms}, we are allowed to exchange summation and limits.\\
It remains to show that the events for $k\ge 2$ are negligible as $n\to\infty$. 
By path-counting, we obtain the bound
\begin{align}\label{sumerror}
 &\sum_{\genfrac{}{}{0pt}{2pt}{|z_i|<\sqrt{\log \Tt}}{i=1,\ldots k} }P_0 \left(H_{a^++Nz_i^+}<{s\Tt},\forall i=1,\ldots, k\right) \nonumber\\
&\qquad\le \sum_{\genfrac{}{}{0pt}{2}{z_i\in\Z^2\text{ distinct } }{i=1,\ldots k} }\prod_{i=1}^{k}P_{Nz_{i-1}^+}\left(H_{\left(Nz_{i}\right)^+}\le s\Tt\right)\, ,
\end{align}
where $Nz_0=a$ and multiplicities arise from the fact that the second sum is over all $z_i\in \Z^2$ distinct. Indeed, we have that
\begin{equation*}
\begin{split}
&P_0 \left(H_{a^++Nz_i^+}<{s\Tt},\forall i=1,\ldots, k\right)\\
& \le\sum_{f\in \Sfrak_k}P_0\left(H_{a^++Nz_{f(1)}^+}<\ldots <H_{a^++Nz_{f(k)}^+}<s\Tt   \right)\\
&\le \sum_{f\in \Sfrak_k}\sum_{j_1+\ldots j_k\le s\Tt}\prod_{i=1}^kP_{Nz_{f(i-1)}^+}\left(H_{Nz_{f(i)}^+}=j_i  \right)\\
& \le\sum_{f\in \Sfrak_k}\sum_{\genfrac{}{}{0pt}{2}{j_i\le s\Tt}{i=1,\ldots ,k} }\prod_{i=1}^kP_{Nz_{f(i-1)}^+}\left(H_{Nz_{f(i)}^+}=j_i  \right)\, ,
\end{split}
\end{equation*}
where $Nz_{f(0)}=a^+$ and $\Sfrak_k$ is the permutation group of size $k$. The sum over $\Sfrak_k$ represents the multiplicities and thus the bound follows.\\
As all the error terms given in \cite[Theorem 1.6 ]{uchiyama2011}\footnote{see Equation (1.8) in that reference and apply a change of variables $u\mapsto a/u$.} are decreasing in $x$, there exists $C>0$ such that for all $x\in\Z^2,\, n\in \N$ large enough
\begin{equation}\label{errorTerms}
    P_0(H_{x^+}<\floor{s\Tt})\le \frac{C}{\log \Tt}\int_0^{s} \p_t(x) \,\d t\, .
\end{equation}
Now it suffices to note that there exists $C>0$ (potentially different from the previous $C$) such that for all $t>0$ we have that
\begin{align}\label{refff}
    \sum_{x\in \Z^2\setminus \{0\}}\p_t(x)\le C\, .
\end{align}
Indeed, combining this with Equation \eqref{errorTerms} allows to bound Equation \eqref{sumerror} by
\begin{equation*}
    \frac{C^k s^k }{\left(\log \Tt\right)^k}\, , 
\end{equation*}
and thus, by choosing $\log \Tt$ large enough, the sum over $k\ge 2$ is negligible in the limit.\\
This finishes the proof. \qed 
\end{proof}
 
We now continue with the proof of Proposition \ref{PropHard}. Using Lemma \ref{hardscalinglemma} and by dominated convergence, we have that
\begin{equation*}
\begin{split}
& \lim_{n\to \infty} \frac{\tau}{\pn_{\e\Tt}\left(a-b\right)}\int_{0}^{\e}\d s\, \Pn_0(H_{a}<{s\Tt})\frac{\left(\pn_{(\e-s)\Tt+1}(b)-\pn_{(\e-s)\Tt}(b)\right)}{\Tt^{-1}}\\
& =\frac{2\pi}{\p_{\e}^{(N)}(a-b)}\int_{0}^\e \d s\int_0^{s}\d t\, \p^{(N)}_t(a)\left(\partial_s \p^{(N)}_{\e -s}(b)\right)\\
& =\frac{2\pi}{\p_{\e}^{(N)}(a-b)}\int_{0}^\e\d s\, \p^{(N)}_{t}(a)\p^{(N)}_{\e-t}(b) \, .
\end{split}
\end{equation*}
As $\L_N$ is a compact set, the above limit is uniform in $a,b\notin B_\rho$. We finish the proof by noting that by the previous discussion (see also \cite[Equation 2.71]{BBH}), we have
\begin{equation*}
    \lim_{\delta\downarrow 0}\lim_{\rho\downarrow 0}\lim_{n\to \infty}\tau\int_0^\delta \d s\, \Pn_{a}(\sigma={s \Tt})\pn_{(e-s)\Tt}(b)=0\, ,
\end{equation*}
and similarly for the integral over $[\e-\delta,\e]$. This concludes the proof of Proposition \ref{PropHard}.
\end{enumerate}
\end{proof}

\subsection{Approximation by skeleton}\label{sctnApprxo}
In this section, we show that the range of the random walk is exponentially equivalent to the range of $\tau/\e$ independent random walk bridges (each of length $\e\Tt$) whose end points are given by the skeleton walk $\Scal_{n,\e}$.
\begin{proposition}\label{apprxProp}
For all $\delta>0$, we have that
\begin{equation}
    \lim_{\e\downarrow 0}\limsup_{n\to \infty}\frac{1}{\tau}\log \Pn_0\left(\frac{1}{\Tt}\babs{\Rcal_n-\Ene[\Rcal_n]}>\delta\right)=-\infty\, .
\end{equation}
\end{proposition}
\begin{proof}
We remind the reader that for $1\le i\le \tau/\e$ the set $\Wcal_i$ had been defined as 
$$
\Wcal_i=\left\lbrace S_j\colon (i-1)\e\Tt \leq j \leq i\e\Tt \right\rbrace\, ,
$$
with
$$
\Rcal_n=  \#\Big\{\bigcup_{i=1}^{\tau/\e}\Wcal_i\Big\}\, .
$$
Also, given $K>0$, let
\begin{align*}
    \Jcal_{n,\e}^K=\Big\{1\le i\le \frac{\tau}{\e} \colon \frac{1}{\sqrt{T_\tau}}\babs{S_{(i-1)\e T_\tau}-S_{i\e T_\tau}}\le K\sqrt{\e}\Big\}\, .
\end{align*}
Define the corresponding ranges as
\begin{align*}
\Rcal_{n,\e}^K  =\#\Big\{\bigcup_{i\in \Jcal_{n,\e}^K}\Wcal_i\Big\}\, ,\text{ and } 
\widehat{\Rcal_{n,\e}^K} =\#\Big\{\bigcup_{i\notin \Jcal_{n,\e}^K}\Wcal_i\Big\}\, .
\end{align*}
Since $0\le \Rcal_n-\Rcal_{n,\e}^K\le \widehat{\Rcal_{n,\e}^K}$, we have that
\begin{equation}\label{3-terms}
    \Babs{\Rcal_n-\Ene\left[\Rcal_n\right]}\le \Babs{\Rcal_{n,\e}^K-\Ene\left[\Rcal_{n,\e}^K\right]}+\widehat{\Rcal_{n,\e}^K}+\Ene\left[\widehat{\Rcal_{n,\e}^K}\right]\, .
\end{equation}
Thus, it suffices to prove that for all $0<\delta<1,\, K>K_0(\delta)$, 
\begin{equation}\label{FirstPPart}
    \lim_{\e \downarrow 0}\limsup_{n\to\infty}\frac{1}{\tau}\log \Pn_0\left(\frac{1}{T_\tau}\Babs{\Rcal_{n,\e}^K-\Ene\left[\Rcal_{n,\e}^K \right]}>\delta\right)=-\infty \,,
\end{equation}
and for all $0<\delta<1,\, K>K_0(\delta)$,
\begin{equation}\label{SecondPPart}
    \lim_{\e \downarrow 0}\limsup_{n\to\infty}\frac{1}{\tau}\log \Pn_0\left(\frac{1}{T_\tau}\widehat{\Rcal_{n,\e}^K}>\delta\right)=-\infty \,.
\end{equation}
Indeed, note that for the third term in \eqref{3-terms} one has $\frac{1}{\Tt} \widehat{\Rcal_{n,\e}^K}\leq\frac{1}{\Tt}\abs{\L_N\cap \Tt^{-1/2}\Z^2}\leq N^2$, and thus
\begin{equation*}
\begin{split}
&\Ene\left[ \frac{1}{\Tt}\widehat{\Rcal_{n,\e}^K} \right]\\
\leq& \Ene\left[ \frac{1}{\Tt}\widehat{\Rcal_{n,\e}^K}\mathbbm{1}\Big{\lbrace \frac{1}{\Tt} \widehat{\Rcal_{n,\e}^K}<\frac{\delta}{2} \Big\rbrace} \right] +
\Ene\left[ \frac{1}{\Tt}\widehat{\Rcal_{n,\e}^K}\mathbbm{1}{\Big\lbrace \frac{1}{\Tt} \widehat{\Rcal_{n,\e}^K}\geq\frac{\delta}{2} \Big\rbrace} \right] \\
\leq& \frac{\delta}{2} + N^2 \Pne\left( \frac{1}{\Tt}\widehat{\Rcal_{n,\e}^K}\geq\frac{\delta}{2} \right)\,.
\end{split}
\end{equation*}
Using the above and then employing the Markov inequality, we obtain
\begin{align*}
\Pn_0\left( \Ene\left[ \frac{1}{\Tt}\widehat{\Rcal_{n,\e}^K} \right]\geq \delta \right)
\leq&\Pn_0\left( \Pne\left( \frac{1}{\Tt}\widehat{\Rcal_{n,\e}^K}\geq\frac{\delta}{2} \right)\geq \frac{\delta}{2N^2} \right)\\
\leq& \frac{2N^2}{\delta} \Pn_0\left( \frac{1}{\Tt}\widehat{\Rcal_{n,\e}^K}\geq\frac{\delta}{2} \right)\,,
\end{align*}
which is controlled analogously to \eqref{SecondPPart}.
\begin{enumerate}
\item \begin{proof}[of Equation \eqref{SecondPPart}]
Applying the exponential Chebyshev inequality and using the definition of the set $\Wcal_i$, as well as the Cauchy-Schwarz inequality, we get
\begin{equation*}
\begin{split}
&\Pn_0\left(\frac{1}{\Tt}\widehat{\Rcal_{n,\e}^K}>\delta\right)\\
\leq& \exp\left(-\frac{\delta \tau}{2 \e} \right) \En_0\left[ \exp\left( \frac{\tau}{2\e}\frac{1}{\Tt}\widehat{\Rcal_{n,\e}^K} \right)\right] \\
\leq& \exp\left(-\frac{\delta \tau}{2 \e} \right) \prod_{i=1}^{\tau/\e} \En_0\left[ \exp\left( \frac{\tau}{2\e}\frac{1}{\Tt}  \#\lbrace \Wcal_i\rbrace\mathbbm{1}{\lbrace i\notin\Jcal_{n,\eps}^K \rbrace}  \right)\right] \\
\leq& \exp\left(-\frac{\delta \tau}{2 \e} \right) \left(1+\sqrt{\delta_K \Kcal_{\tau,\e}}\right)^{\tau/\e}\,  ,
\end{split}
\end{equation*}
where we compress the notation
\begin{align*}
\delta_K &=\sup_n\Pn_0\left(\frac{1}{\sqrt{T_\tau}}\abs{S_{\e \Tt}}>K\sqrt{\e}\right),\\
\Kcal_{\tau,\e} &=\En_0\left[\exp\left(\frac{\tau}{\e T_\tau}\#\{\Wcal_1\}\right)\right].
\end{align*}
In particular, by the Donsker's invariance principle, we have $\delta_K=o(1)$ as $K\to \infty$. Thus, it still remains to show 
\begin{equation}\label{constantC}
    \limsup_{n\to\infty} \En_0\left[\exp\left( \frac{\tau}{\e\Tt }\#\{\Wcal_1\}\right)\right]<\infty\, .
\end{equation}
Recall the definition of $\Wcal_1=\left\lbrace S_j\colon 0\leq j \leq \e\Tt \right\rbrace$, it follows
\begin{align*}
\En_0\left[\exp\left( \frac{\tau}{\e\Tt }\#\{\Wcal_1\}\right)\right] 
=\En_0\left[\exp\left( \frac{\tau}{\e\Tt }\Rcal_{\e\Tt }\right)\right] \,.
\end{align*}
Furthermore, from \cite[Theorem 6.3.1]{chen2010random}, we know that for every $\theta>0$ fixed
\begin{equation}\label{CheResult}
\sup_{m\ge 1}E_0^{\langle m\rangle}\left[ \exp\left(\theta\,\frac{\log m}{m }\Rcal_m\right)\right]
=:C<\infty\, .
\end{equation}
Substituting $m=\e\Tt$, Equation \eqref{constantC} is verified. 
Finally, we obtain
\begin{align*}
&\lim_{\e \downarrow 0}\limsup_{n\to\infty}\frac{1}{\tau}\log \Pn_0\left(\frac{1}{T_\tau}\widehat{\Rcal_{n,\e}^K}>\delta\right) \\
\leq& \lim_{\e \downarrow 0}\limsup_{n\to\infty}\frac{1}{\tau}\left(-\frac{\delta \tau}{2 \e} + \log\left( 1+\sqrt{\delta_K \Kcal_{\tau,\e}} \right)^{\tau/\e}\right) \\
\leq& \lim_{\e \downarrow 0}\limsup_{n\to\infty} -\frac{\delta }{2 \e}+\frac{1}{\e}\sqrt{C\delta_K}\,,
\end{align*}
and thus, the desired result follows by taking $K$ sufficiently large.
\end{proof} 
\item \begin{proof}[of Equation \eqref{FirstPPart}]
Similar to \cite{BBH} and \cite{PP11}, it follows from an application of Talagrand's inequality: Let us denote the power-set of a given set $S$ by $\mathfrak{P}(S)$ and define
$
\Tcal=\mathfrak{P}(\Dt)\, .
$
Endow the space of all subsets of $\Dt$ with the metric $d\colon \Tcal\times\Tcal \to \lbrack 0,\infty)$ with
$$
    d(U,V)= \frac{1}{\Tt}\# \lbrace U\Delta V\rbrace=\frac{1}{\Tt} \# \big\lbrace (U\backslash V) \cup (V\backslash U) \big\rbrace.
$$ 
Furthermore, let us define the product space
$$
\Omega=\Tcal^{\otimes (\tau/\e)}\,,
$$
on which the collection of random subsets $(\Wcal_i)_i$, generated by the random walk (which is distributed with respect to $\Pne$), induce a product measure denoted by $\Pne$. For $C=\{C_i\}\in\Omega$, we define
\begin{equation*}
    M(C)=\#\Big\{\bigcup_{i\in\Jcal_{n,\e}^K}C_i\Big\}\, .
\end{equation*}
Let us denote the median of $\Rcal_{n,\e}^K =\#\Big\{\bigcup_{i\in \Jcal_{n,\e}^K}\Wcal_i\Big\}$ under $\Pne$ by $m$, that is to say,  $m=\inf\lbrace i:\Pne \left( \Rcal_{n,\e}^K <i \right)\geq \frac{1}{2} \rbrace$, and then define the event
\begin{equation*}
    A=\{C\in\Omega\colon M(C)\le m\}\, .
\end{equation*}
Note that by the definition of $A$, there exists $\xi>0$ such that for large enough $n$ it holds $0<\xi<\Pne(A)<1-\xi<1$. 
For any $\l>0$, applying Talagrand's inequality \cite[Theorem 2.4.1]{Talagrand95} provides the bound 
\begin{equation}\label{TalagrandBound}
\Ene\left[ \ex^{\l f \left(A,\{W_i\} \right)}\right]
\le \xi^{-1}\prod_{i\in \Jcal_{n,\e}^K}\Ene\left[\cosh\left({\frac{\l}{\Tt}\#\{\Wcal_i\Delta\Wcal_i'\}}\right)\right]\, ,
\end{equation}
where 
\begin{equation*}
f \left(A,\{C_i\} \right)=\inf_{C'\in A} \sum_{i\in \Jcal_{n,\e}^K}\frac{\l}{\Tt}\#(C_i\Delta C_i')\, ,
\end{equation*}
and $\Wcal_i'$ denotes an i.i.d copy of $\Wcal_i$. Combining \eqref{TalagrandBound} and the exponential Chebyshev inequality, we arrive at
\begin{equation}\label{TalagrandBound2}
\begin{split}
&\Pne(f(A,\{W_i\})\ge \delta) \\
\le &\xi^{-1}\inf_{\l>0}\ex^{-\l \delta}\prod_{i\in \Jcal_{n,\e}^K}\Ene\left[\cosh\left({\frac{\l}{\Tt}\#\{\Wcal_i\Delta\Wcal_i'\}}\right)\right]
:=\Xi(\delta)\, .
\end{split}
\end{equation}
By symmetry, one can repeat the argument of \eqref{TalagrandBound} and \eqref{TalagrandBound2} for $\hat{A}=\{C\in\Omega\colon M(C) > m\}$ and obtain
\begin{align}
\Pne &\left(\frac{1}{\Tt}\babs{\Rcal_{n,\e}^K-m}\ge \delta\right)\\
\le& \Pne(f(A,\{W_i\})\ge \delta)+\Pne(f(\hat{A},\{W_i\})\ge \delta) \nonumber\\
\le& \,2\Xi(\delta)\, .
\end{align}
We can bound
\begin{equation}\label{EERneMedian}
    \frac{1}{\Tt}\Babs{\Ene \left[\Rcal_{n,\e}^K\right]-m}\le \frac{\delta}{3}+\frac{1}{\Tt}\#\Dt \Pne\left(\frac{1}{\Tt}\Babs{ \Rcal_{n,\e}^K -m}\ge \delta/3\right)\, .
\end{equation}
Using $\frac{1}{\Tt}\#\Dt\le N^2$ , Equation \eqref{TalagrandBound2} and Equation \eqref{EERneMedian} we obtain
\begin{equation*}
\begin{split}
&\Pne\left(\frac{1}{\Tt}\Babs{\Rcal_{n,\e}^K-\Ene\left[\Rcal_{n,\e}^K\right]}\ge \delta\right)\\
\le& \,\Pne\left( \frac{1}{\Tt}\Babs{\Rcal_{n,\e}^K-m}\ge \delta/3\right) 
+ \1\Big\{\frac{1}{\Tt}\Babs{\Ene\left[\Rcal_{n,\e}^K\right]-m}\ge 2\delta/3\Big\}\\
\le& \,2\Xi(\delta/3)+\1\Big\{\Pne\left(\frac{1}{\Tt}\babs{\Rcal_{n,\e}^K-m} \ge \delta/3 \right)\ge \tfrac{2\delta}{3N^2}\Big\}\\
\le& \,2\Xi(\delta/3)+\1\Big\{2\Xi(\delta/3)\ge \tfrac{2\delta}{3N^2}\Big\}\, .
    \end{split}
\end{equation*}
Using the Markov inequality and taking the expectation with respect to $\Pn_0$, we arrive at
\begin{equation*}
\begin{split}
&\Pn_0\left(\frac{1}{T_\tau}\Babs{\Rcal_{n,\e}^K-\Ene\left[\Rcal_{n,\e}^K \right]}>\delta\right)\\
\le &\,\En_0\left[2\Xi(\delta/3)\right] + \Pn_0\left( \Xi(\delta/3)\ge \tfrac{\delta}{3N^2} \right)\\
\le & \,2\left(1+\frac{3N^2}{\delta}\right)\En_0\big[2\Xi(\delta/3)\big]\, .
\end{split}
\end{equation*}
The proof of \eqref{FirstPPart} has been reduced to showing that for $0<\delta<1$
\begin{equation*}
\lim_{\e\downarrow 0}\lim_{n\to\infty}\frac{1}{\tau}\log \En_0\big[2\Xi(\delta/3)\big]=-\infty\, .
\end{equation*}
Notice that the definition of $\Rcal_{n,\e}^K$ enforces constraints on the range of $\Scal_{n,\e}$. Keeping this in mind, we apply the inequality $\cosh(\l x)\leq 1+\l^2\exp(x)$ for $0<\lambda \leq 1$ and $x>0$ to bound for $\l=c\tau/\e$  with $0<c<1$
\begin{equation*}
\begin{split}
\Ene &\left[\cosh\left({c\frac{\tau}{\e\Tt}\#\{\Wcal_i\Delta\Wcal_i'\}}\right)\right] \\
\le& 1+c^2\Ene\left[\exp\left({\frac{\tau}{\e\Tt}\#\{\Wcal_i\Delta\Wcal_i'\}}\right)\right]\\
\le& 1+c^2\left(\Ene\left[\exp\left({\frac{\tau}{\e\Tt}\#\{\Wcal_i\}}\right)\right]\right)^2\\
\le& 1+c^2\left(\sup_{|x|\le K}E_{0,x^+}\left[\exp\left({\frac{\tau}{\e\Tt}\#\{\Wcal_i\}}\right)\right]\right)^2\, ,
\end{split}
\end{equation*}
where in the last step we bounded the range of a random walk bridge on $\Dt$ by the range of a random walk bridge on $\Z^d$. By Lemma \ref{brdgeLmma}, there is a constant $C_K$ such that the above expression is bounded by $1+c^2C_K$. This gives us the bound of the right-hand-side of \eqref{TalagrandBound2}, i.e.
\begin{equation*}
\Xi(\delta)\le 2 \ex^{-c\delta\frac{\tau}{\e}}(1+c^2C_K^2)^{\tau/\e}
\le 2\ex^{(-c\delta+c^2C_K^2)\tau/\e}\big|_{c=\delta/2C_K^2}
= 2\ex^{(-\delta/4C_K^2)\tau/\e}\, .
\end{equation*}
Since $C_K\ge 1$, for $\delta$ small enough, we finally arrive at
\begin{equation*}
\lim_{\e\downarrow 0}\lim_{n\to\infty}\frac{1}{\tau}\log \En_0\big[2\Xi(\delta/3)\big]=-\infty\, .
\end{equation*}
\end{proof}
\end{enumerate}
Since we have shown both \eqref{FirstPPart} and \eqref{SecondPPart}, the proof of Proposition \ref{apprxProp} (establishing exponential equivalence) is completed.
\end{proof}

\subsection{An LDP of range on $\L_N$}\label{sctnLDPLN}
We show that the range of random walk wrapped around $\Dt$ satisfies an LDP: 
\begin{proposition}\label{prop_LDP}
Under $\Pn_0$, the random variable $\frac{1}{\Tt}\Rcal_n$ satisfies an LDP on $\R^+$ with rate $\tau$ and rate function $J_N$, where
\begin{align}\label{def_LDPratefct}
J_N(b)=\inf_{\phi\in\partial\Phi_N^{2\pi}(b)}\left\lbrack \frac{1}{2}\int_{\L_N} |\nabla\phi|^2(x)\mathrm{d}x \right\rbrack
\end{align}
with
\begin{align}\label{def_LDPdomainInf}
\partial\Phi_N^{2\pi}(b)=\left\lbrace \phi\in H^1(\L_N)\colon \int_{\L_N} \phi^2(x) \mathrm{d}x=1,\;\int_{\L_N} (1-e^{-2\pi\phi^2(x)})\ \mathrm{d}x=b \right\rbrace.
\end{align}
\end{proposition}
\begin{proof}
Employing Proposition \ref{sceletonLDP}, Proposition \ref{apprxProp} and Varadhan's lemma, for any bounded continuous function $f\colon \R^+\to \R$, we have that
\begin{equation*}
\begin{split}
\lim_{\tau\to\infty}&\frac{1}{\tau}\log\En_0\left[\exp\left(\frac{1}{\Tt}f\left(\Tt^{-1}\Rcal_n\right)\right)\right] \\
=& \lim_{\e\downarrow 0}\lim_{\tau\to\infty}\frac{1}{\tau}\log\En_0\left[\exp\left(\frac{1}{\Tt}f\left(\Ene\left[\Tt^{-1}\Rcal_n\right]\right)\right)\right]\\
=& \lim_{\e\downarrow 0}\sup_{\mu \in \Mcal_1(\L_N\times\L_N)}\left\{f\left(\Phi_{1/\e}(\mu)\right)-\e^{-1}I^{(2)}_{\e}(\mu)\right\}\, .
\end{split}
\end{equation*}
Repeating the argument in \cite[(2.95), Lemma 5-7]{BBH} and rescaling, we get that 
\begin{align*}
\lim_{\e\downarrow 0}&\sup_{\mu \in \Mcal_1(\L_N\times\L_N)}\left\{f\left(\Phi_{1/\e}(\mu)\right)-\e^{-1}I^{(2)}_{\e}(\mu)\right\}\\
 =&\sup_{\substack{\phi\in H^1(\L_N)\\ \|\phi\|_2=1}}\Big\{f\left(\int_{\L_N}\d x \left(1-\ex^{-2\pi \phi^2(x)}\right)\right)-\frac{1}{2}\norm{\nabla \phi}_2^2\Big\}\, .
\end{align*}
The rest of the proof follows directly by applying the inverse of Varadhan's lemma, see \cite{dembo2009large}.
\end{proof}

\subsection{Proof of Theorem \ref{thm_main}}
We begin with the upper bound:
\begin{equation*}
    \limsup_{n\to\infty}\frac{1}{\tau}\log P_0(\Tt^{-1}\Rcal_n\le b)\le \lim_{n\to\infty}\frac{1}{\tau}\log\Pn_0(\Tt^{-1}\Rcal_n\le b)=-J_N(b)
\end{equation*}
By \cite[Proposition 2]{BBH}, we have that $\lim_{N\to\infty}J_N=I$, where $I$ is the rate function given in \eqref{def_ratefct}, and the upper bound then follows.\\
For the lower bound, let $\Ccal$ be the event that the (rescaled) random walk does not hit the boundary of $[-N/2,N/2)$ up to time $n$. We then have that
\begin{equation*}
    P_0\left(\Tt^{-1}\Rcal_n\le b\right)\ge\Pn_0\left(\Tt^{-1}\Rcal_n\le b,\,\Ccal \right)\, .
\end{equation*}
By repeating the steps in subsections 3.2-3.5, one can show that
\begin{equation*}
    \lim_{n\to\infty}\frac{1}{\tau}\log\Pn_0\left(\Tt^{-1}\Rcal_n\le b|\,\Ccal\right)=-J_N(b)\, .
\end{equation*}
By Brownian approximation, as done in Section \ref{DskrVrdhSubsctn}, one can show that
\begin{equation*}
    \lim_{N\to\infty}\liminf_{n\to\infty}\frac{1}{\tau}\log\Pn_0(\Ccal)\ge 0\, , 
\end{equation*}
and thus, the proof of Theorem \ref{thm_main} is completed.

\section{Appendix}
\subsection{Properties of the rate function}
Various properties of the rate function $I^{g_a}$ in \eqref{BBHbound} were established in \cite{BBH}. The results are applicable to our rate function $I$ in \eqref{def_ratefct} due to the fact that the function $I$ agrees with $I^{g_a}$ up to a multiplicative constant. We only give the statement and refer the reader for the complete proof to the original paper. 
\begin{theorem}\cite[Theorem 3, 4]{BBH}\label{ThmPropRateFunct}
\begin{enumerate}
    \item 
For every $b>0$, $I(b)=(4\pi)^{-1}\chi\left(b/(2\pi)\right)$, where $\chi\colon (0,\infty)\to [0,\infty)$ and
\begin{equation}\label{minProb}
    \chi(u)=\inf\Big\{\norm{\nabla \psi}_2^2\colon \psi \in H^1(\R^2),\, \norm{\psi}_2=1,\, \int_{\R^2}\left(1-\ex^{-\psi^2}\right)\le u\Big\}\, .
\end{equation}
\item
$\chi$ is continuous on $(0,\infty)$ and strictly decreasing on $(0,1)$. Furthermore, $\chi(u)=0$ for $u\ge 1$.
\item
$u\mapsto u\chi(u)$ is strictly decreasing on $(0,1)$ and
\begin{equation*}
    \lim_{u\downarrow 0}u\chi(u)=-\l_2\, ,
\end{equation*}
where $\l_2$ is the smallest Dirichlet eigenvalue of $-\Delta$ on a ball of unit volume. 
\item $u\mapsto(1-u)^{-1}\chi(u)$ is strictly decreasing in $(0,1)$ and
\begin{equation*}
    \lim_{u\uparrow 1}(1-u)^{-1}\chi(u)=2\mu_2\, ,
\end{equation*}
with
\begin{equation*}
    \mu_2=\inf\{\norm{\nabla \psi}_2^2\colon \psi \in H^1(\R^2)\colon \norm{\psi}_2=1,\,\norm{\psi}_4=1\}\, .
\end{equation*}
satisfying $0<\mu_2<\infty$.
\item
For all $u\in (0,1)$, the minimization problem \eqref{minProb} has at least one minimiser. It is strictly positive, radially symmetric and strictly decreasing in the radial component. All other minimisers are of the same type.
\end{enumerate}
\end{theorem}

\subsection{Technical supplement}
We collect some technical results we have used in previous sections.
\begin{lemma}\label{BscBrwnApprxLemma}
Fix $N>0$. There exists some $\delta>0$ such that we have for $a,b\in \L_N$ and $\e>0$ 
\begin{equation}
\pn_{\e\Tt}(a,b)
=\frac{1}{\Tt}\p^{(N)}_\e(a,b)+\Ocal\left(\frac{1}{\Tt^{1+\delta}}\right)\, .
\end{equation}
\end{lemma}
\begin{proof}
For any $a,b\in \L_N$, let us set $m=\Tt^{1/4+\delta}$ for some $\delta>0$ (on which we will put some constraints later), we decompose the probability
\begin{equation*}
\begin{split}
\pn_{\e\Tt}(a,b)
=& \sum_{ \substack{z\in\Z^2 \\|z|\le m}} p_{\e\Tt}\left({a}^+,{b}^++ {z^+}N \right) 
+ \sum_{ \substack{z\in\Z^2 \\|z|>m}} p_{\e\Tt}\left({a}^+,{b}^++ {z^+}N \right)  \\
=& \sum_{ \substack{z\in\Z^2 \\|z|\le m}} P_{{a}^+}(S_{2\e\Tt}=b^++z^+N) + \Ecal\, ,
\end{split}
\end{equation*}
with $\Ecal\le P_0\left(\max_{0\le j\le \e\Tt}|S_j|\ge mN\Tt^{1/2}\right)$.
For each summand in the first term, the transition probability of going from the point ${a}^+$ to $b^++ {z^+}N$ on $\Z^2$ can be estimated by the local central limit theorem in \eqref{LCLT}
\begin{align*}
&p_{\e\Tt}\left({a}^+,{b}^++{z^+}N \right)\\
=& \frac{1}{\Tt}\p_\e(a,b+Nz)\left((1+\Ocal(\Tt^{-1/2})\right)+A_{\e\Tt }(a^+, b^++ {z^+}N)\, ,
\end{align*}
where $\Ocal-$term arises from the rounding issue. Note that $A_t(x,y)<1/t^2$ uniformly in all $t>0$. Therefore,
\begin{equation*}
    \sum_{z\in\Z^2\colon |z|\le m}A_{\e\Tt }(a^+, b^++{z^+}N)\, .
\end{equation*}
is at most $\Ocal(\Tt^{-3/2+2\delta})$. \\
For the second term, by the assumption that $E\left[|X_i|^4\right]<\infty$ we can bound by \cite[Equation 2.6]{Lawler2010}
\begin{equation*}
    P_0\left(\max_{0\le j\le \e\Tt}|S_j|\ge mN\Tt^{1/2}\right)=\Ocal\left(\frac{1}{\Tt^{1+4\delta}}\right)\, ,
\end{equation*}
as $n\to\infty$. Now, choosing $\delta$ sufficiently small yields the claim.
\end{proof}

\begin{lemma}\label{brdgeLmma}
We have that for every $K>0$ and $\e>0$
\begin{equation}
    \sup_{n\in \N}\sup_{|x|\le K}E_{0,x^+}\left[\exp\left({\frac{\tau}{\e\Tt}\Rcal_{\e\Tt}}\right)\right]:=C_K<\infty\, ,
\end{equation}
with 
\begin{equation*}
    \limsup_{K\to\infty} C_K> 0\, .
\end{equation*}
\end{lemma}
\begin{proof}
The reasoning here follows \cite{BBH}. The second statement is trivial, as $C_K\ge 1$ by construction. 
Denote by $E_{0,y,x}$ the measure of the random walk conditioned to be at $y$ at time $\e\Tt/2$ and at $x$ at time $\e\Tt$. We then have
\begin{equation*}
\begin{split}
& E_{0,x^+}\left[\exp\left({\frac{\tau}{\e\Tt}\Rcal_{\e\Tt}}\right)\right]\\
&= \sum_{y\in \Z^d}\frac{p_{\e\Tt/2}(0,y)p_{\e\Tt/2}(y,x^+) }{p_{\e\Tt}(0,x^+)} E_{0,y,x^+}\left[\exp\left({\frac{\tau}{\e\Tt}\Rcal_{\e\Tt}}\right)\right]\\
&\le \sum_{y\in \Z^d}\frac{p_{\e\Tt/2}(0,y)p_{\e\Tt/2}(y,x^+) }{p_{\e\Tt}(0,x^+)} E_0\left[\ex^{{\frac{\tau}{\e\Tt}\Rcal_{\e\Tt/2}}}\Big | S_{\e\Tt/2}=y\right]\\
&\qquad\qquad\qquad\qquad\qquad\qquad\qquad\times E_{0}\left[\ex^{{\frac{\tau}{\e\Tt}\Rcal_{\e\Tt/2}}}\Big | S_{\e\Tt/2}=y-x^+\right]\, ,
\end{split}
\end{equation*}
where in the last step we used the subadditivity of $\Rcal_n$. Then apply Cauchy-Schwarz inequality and Jensen's inequality to get
\begin{align*}
\le& \frac{1}{p_{\e\Tt}(0,x^+)}\sum_{y\in \Z^d}\left( p_{\e\Tt/2}(0,y)E_0\left[\exp\left({\frac{\tau}{\e\Tt}\Rcal_{\e\Tt/2}}\right)\Big | S_{\e\Tt/2}=y\right]\right)^2 \\
\le& \frac{1}{ p_{\e\Tt}(0,x^+)}\sum_{y\in \Z^d} p_{\e\Tt/2}(0,y)^2E_0\left[\exp\left({\frac{2\tau}{\e\Tt}\Rcal_{\e\Tt/2}}\right)\Big | S_{\e\Tt/2}=y\right]\,.
\end{align*}
Employing \cite[Proposition 2.4.6]{Lawler2010}, we can bound $p_t(x,y)\le c/t$ for some $c>0$ neither depending on $t\in \N$ nor $x,y \in \Z^2$. Applying \eqref{CheResult} and (in the second step) the Markov property, we conclude that the quantity above is bounded from above by
\begin{equation*}
\begin{split}
\le& \frac{2c }{\e\Tt p_{\e\Tt}(0,x^+)}\sum_{y\in \Z^d} p_{\e\Tt/2}(0,y)E_0\left[\exp\left({\frac{2\tau}{\e\Tt}\Rcal_{\e\Tt/2}}\right)\Big | S_{\e\Tt/2}=y\right]
\\
\le& \frac{2c }{\e\Tt p_{\e\Tt}(0,x^+)}E_0\left[\exp\left({\frac{2\tau}{\e\Tt}\Rcal_{\e\Tt/2}}\right)\right]\\
\le& 2c \ex^{K^2/(2\e)}E_0\left[\exp\left({\frac{2\tau}{\e\Tt}\Rcal_{\e\Tt/2}}\right)\right]\, .
\end{split}
\end{equation*}
This finishes the proof.
\end{proof}


\section{Notation Glossary}\label{glossary}
In this section, we list most of the notation used throughout the paper.
\subsection{Spaces and Projections}
We will always work with the Skorokhod space $D$ of cadlag paths from $[0,\infty)$ onto $\R^2$. The family of coordinate projections on $\R^2$ is denoted by $(S_t)_{t\ge 0}$. Given a space $E$, denote the space of probability measures of $E$ by $\Mcal_1(E)$.

\subsection{Domains and Scalings}
In general, throughout the paper, $N$ is a cut-off constant, compactifying $\R^2$ to the torus. Furthermore, $n$ is the inverse lattice spacing.
\begin{enumerate}
    \item
    $x^+=\floor{x\Ttt}$ and $x^-=\floor{x\Ttt}\Tt^{-1/2}$ for $x\in \R^2$.
    \item
    $\tau=\log n$, the scaling of the LDP.
    \item
    $\Tt=n/\log n$, the mean scaling.
    \item
    $Q_\tau=\sqrt{\log \Tt\log\log \Tt}^{-1}$, the size of the holes during the cutting procedure.
    \item 
    $\L_N=[-N/2,N/2)^2$, the continuum torus of length $N$.
    \item
    $\Delta_\tau=\L_N\cap \Tt^{-1/2}\Z^2$, the rescaled lattice.
    \item
    $B_{R}=\{x\in \R^2\colon |x|\le r\}$, the ball centered at zero, also $B_R(x):=x+B_R$.
\end{enumerate}

\subsection{Kernels and Measures}
The superscripts $(N)$ and $\langle N\rangle$ imply that the underlying process lives on the torus. The use of \texttt{mathfrak} and \texttt{mathbb} indicate continuum objects. We occasionally use the shorthand notation $p(y-x)=p(x,y)$, if a kernel $p$ is translation invariant. Similarly, if the sub/superscript is equal to zero, occasionally we omit it, i.e. $P=P_0$.
\begin{enumerate}
    \item
    $P_x$ measure of the planar random walk defined on $\Z^2$ started at $x$, and $p_t(x,y)$ the transition kernel from point $x$ to point $y$ at time $t$ associated to $P_x$.
    \item
    $\Pn_x$ measure of the planar random walk projected into $\Dt$, and $p^{\langle N\rangle}_t(x,y)$ the transition kernel associated to $\Pn_x$.
    \item 
    $\P_x$ the measure of Brownian motion defined on $\R^2$, and $\p_t(x,y)$ the Brownian transition kernel associated to $\P_x$.
    \item
    $\P_x^{(N)}$ the measure of Brownian motion projected onto $\L_N$, and $\p^{(N)}_t(x,y)$ the transition kernel associated to $\P_x^{(N)}$ .
    \item
    $\P_{a,b}$, $\P_{a,b}^{(N)}$, $P_{a,b}$, $P_{a,b}^{\langle N\rangle }$ the bridge measures (on the whole space and on the torus) of length $\e$ for the Brownian motion, and length $\e\Tt$ for the random walks. The expectation is denoted in the same style.
    \item
    Define $\Scal_{n,\eps}=\{S_{i\eps \Tt}\}_{1\le i\le \tau/\eps}$ the skeleton walk, then denote $\Pne,\Ene$ the conditional law/expectation given $\Scal_{n,\e}$ (where $S_t$ is distributed under $\Pn_0$).
    \item
    $b_{n,\e}(y,z)=\Pn_{y^-}\left(\sigma\le \e\Tt\mid S_{\e \Tt}=z^-\right) $, and $b_{n,\e}^{\rho}(y,z)=b_{n,\e}(y,z)\1\{y,z\notin B_{\rho}\}$ for $y,z\in \R^2$ with $\rho>0$ the radius of the centred ball. 
    \item
    $\phi_\e(y,z)=\left[{\p_{\e/2}^{(N)}(z-y)}\right]^{-1}{\int_0^\e \d s \p_{s/2}^{(N)}(-y)\p_{(\e-s)/2}^{(N)}(z)}$ with \newline
    $\phi_\e^\rho(y,z)=\phi_\e(y,z)\1\{y,z\notin B_{\rho}\}$.
\end{enumerate}
\subsection{Stopping times}
\begin{enumerate}
    \item 
    $\sigma=\min\{n\ge 0 \colon S_n=0\}$ and $\sigma_r=\min\{n\ge 0\colon S_n\in\partial B_{N/r}\}$.
    \item
    $H_z=\min\{n\ge 0\colon S_n=z\}$, for a $z\in \Z^2$.
\end{enumerate}

\section*{Acknowledgements}
The authors would like to thank the two anonymous referees whose valuable suggestions improved the paper greatly.\\
Quirin Vogel would like to thank his supervisors Stefan Adams and Wei Wu for their support. Research of Jingjia Liu was funded by the Deutsche Forschungsgemeinschaft (DFG, German Research Foundation) under Germany's Excellence Strategy EXC 2044 -390685587, Mathematics M\"unster: Dynamics -Geometry -Structure.

\end{document}